\documentclass[12pt,reqno]{amsart}
\usepackage{amssymb}
\usepackage{mathrsfs}
\usepackage{a4}
\usepackage{hyperref}

\newcommand{\heute}{28 April 2010}

\theoremstyle{plain}
\newtheorem{theorem}{Theorem}[section]

\newtheorem{lemma}[theorem]{Lemma}
\newtheorem{corollary}[theorem]{Corollary}
\newtheorem{proposition}[theorem]{Proposition}
\newtheorem{conjecture}[theorem]{Conjecture}
\theoremstyle{remark}
\newtheorem{remark}[theorem]{Remark}

\newtheorem{hypothesis}[theorem]{Hypothesis}
\newtheorem*{defn}{Definition}
\newtheorem*{rk}{Remark}
\newtheorem*{bsp}{Example}
\newtheorem*{notn}{Notation}

\newcommand{\enref}[1]{\textup{(\ref{enum:#1})}}
\newcommand{\dashTwo}[1]{\textup{(\ref{two}${}'$)}}

\newcommand{\mytag}[1]{\refstepcounter{theorem}\label{#1}\tag{\thetheorem}}

\newcommand{\ignore}[1]{}

\newcommand{\f}[1][p]{\mathbb{F}_{#1}}

\newcommand{\Gro}[1]{Gr\"ob\-ner}

\DeclareMathOperator{\End}{End}

\DeclareMathOperator{\rank}{rank}

\newcommand{\GL}{\mathit{GL}}

\newcommand{\xx}{\mathfrak{X}}
\newcommand{\yy}{\mathcal{Y}}

\newcommand{\abs}[1]{\left|#1\right|}
\newcommand{\scr}[1]{\mathscr{#1}}
\newcommand{\partref}[1]{(\ref{#1})}

\begin{document}

\title[Strong form of Oliver's conjecture]{On a strong form of Oliver's
$p$-group conjecture}
\author[D.~J. Green]{David J. Green}
\address{Department of Mathematics \\
Friedrich-Schiller-Universit\"at Jena \\ 07737 Jena \\ Germany}
\email{david.green@uni-jena.de}
\author[L.~H\'ethelyi]{L\'aszl\'o H\'ethelyi}
\address{Department of Algebra \\
Budapest University of Technology and Economics \\ Budapest \\ Hungary}
\thanks{H\'ethelyi supported by the Hungarian Scientific Research Fund
OTKA, grant 77-476}
\author[N.~Mazza]{Nadia Mazza}
\address{Department of Mathematics and Statistics \\ Fylde College \\
Lancaster University \\ Lancaster LA1 4YF \\ United Kingdom}
\date{\heute}

\begin{abstract}
We introduce a strong form of Oliver's $p$-group conjecture and derive a
reformulation in terms of the modular representation theory of a quotient
group. The Sylow $p$-subgroups of the symmetric group $S_n$ and of the
general linear group $\GL_n(\f[q])$ satisfy both the strong conjecture and
its reformulation.
\end{abstract}

\maketitle

\section{Introduction}
\label{intro}
\noindent
Bob Oliver proved the Martino--Priddy Conjecture, which states that
the $p$-local fusion system of a finite group~$G$ uniquely determines
the $p$-completion of its classifying
space~$BG$~\cite{Oliver:MartinoPriddyOdd,Oliver:MartinoPriddyEven}\@.
Finite groups are not however the only source of fusion systems: the $p$-blocks
of modular representation theory have fusion systems too, and there are
other exotic examples as well. An open question in the theory of
$p$-local finite groups claims that every fusion system has a
unique $p$-completed classifying space --
see~\cite{BLO:survey} for a survey article on this field.

In~\cite{Oliver:MartinoPriddyOdd}, Oliver introduced the characteristic
subgroup $\xx(S)$ of a finite $p$-group~$S$\@. For odd primes he showed that
unique existence of the classifying space would follow from the
following conjecture.
Recall that $J(S)$ denotes the Thompson subgroup of~$S$ generated by all
elementary
abelian subgroups of greatest rank.

\begin{conjecture}[Oliver; Conjecture~3.9 of \cite{Oliver:MartinoPriddyOdd}]
\label{conj:Oliver}
Let $p$ be an odd prime and $S$ a finite $p$-group. Then $J(S) \leq \xx(S)$.
\end{conjecture}

\noindent
In~\cite{oliver}, the first two authors and Lilienthal obtained the following reformulation
of Oliver's conjecture.

\begin{conjecture}[Conjecture~1.3 of \cite{oliver}]
\label{conj:GHL}
Let $p$ be an odd prime and $G$ a finite $p$-group. If the faithful $\f G$ module $V$ is an
$F$-module, then there is an element $1 \neq g \in \Omega_1(Z(G))$ such that the minimal
polynomial of the action of $g$~on $V$ divides $(X-1)^{p-1}$.
\end{conjecture}

\noindent
Recall that $V$ is by definition an $F$-module if there is an \emph{offender},
i.e.\@, an elementary abelian subgroup $1 \neq E \leq G$ such that $\dim(V) - \dim(V^E) \leq
\rank(E)$.
Theorem~1.2 of~\cite{oliver} states that Conjecture~\ref{conj:Oliver} holds for every $p$-group~$S$
with $S / \xx(S) \cong G$ if and only if Conjecture~\ref{conj:GHL} holds for~$G$. Note
that by \cite[Lemma~2.3]{oliver}, every $p$-group $G$ does occur as some
$S / \xx(S)$.

In this paper we modify Oliver's construction of $\xx(S)$ slightly to obtain
a new characteristic subgroup $\yy(S)$, with $\yy(S) \leq \xx(S)$
for $p$~odd and $\yy(S) = S$ for $p=2$: see~\S\ref{section:Y}\@.
We therefore propose the following strengthening of Oliver's conjecture:

\begin{conjecture}
\label{conj:Y1}
Let $p$ be a prime and $S$~a finite $p$-group. Then $J(S) \leq \yy(S)$.
\end{conjecture}

\noindent
This is indeed a strengthening: if $p=2$ then $\yy(S) = S$ and so Conjecture~\ref{conj:Y1}
is true, whereas Oliver's conjecture is false in general; and
for $p \geq 3$ we have $\yy(S) \leq \xx(S)$.
We recently learnt that Justin Lynd raises the same question in his
paper~\cite{Lynd:2subnormal}\@; note that his $\xx_3(S)$ is our $\yy(S)$.

This conjecture admits a module-theoretic reformulation as well. An element $g \in G$ is
said to be \emph{quadratic} on the $\f G$-module $V$ if its action has minimal
polynomial $(X-1)^2$. Note that if~$V$ is faithful then quadratic elements must have order~$p$.

\begin{conjecture}
\label{conj:Y2}
Let $p$~be a prime and $G$ a finite $p$-group.
If the faithful $\f G$-module $V$ is an $F$-module,
then there are quadratic elements in $\Omega_1(Z(G))$.
\end{conjecture}

\noindent
Observe that Conjecture~\ref{conj:Y2} is a strengthening of
Conjecture~\ref{conj:GHL}, and trivially true for $p=2$\@.
Our first actual result is the equivalence of the two new conjectures:

\begin{theorem}
\label{thm:Y12}
Let $p$ be an odd prime and $G$ a finite $p$-group.
Then Conjecture~\ref{conj:Y1} holds
for every finite $p$-group $S$ with $S / \yy(S) \cong G$ if and only if
Conjecture~\ref{conj:Y2} holds for every faithful $\f G$-module~$V$.
\end{theorem}

\begin{proof}
The proof of \cite[Theorem~1.2]{oliver} adapts easily to the present case, given the
results on $Y$-series and on $\yy(G)$ in Lemma~\ref{lemma:Yportmanteau}\@.
\end{proof}

\noindent
Note that Conjecture~\ref{conj:Y2} holds if $G$ is metabelian or of (nilpotence) class
at most four: for though \cite[Theorem 1.2]{newol} only purports to verify the weaker
Conjecture~\ref{conj:GHL}, the proof that is given there only uses the assumptions of
Conjecture~\ref{conj:Y2}\@.

Our main result is the following:

\begin{theorem}
\label{thm:ol3main}
For $n \geq 1$ and a prime $p$, let $P$ be a Sylow $p$-subgroup of the
symmetric group~$S_n$.
Then Conjecture~\ref{conj:Y1} holds with $S = P$, and Conjecture~\ref{conj:Y2}
holds with $G = P$.
\end{theorem}

\begin{proof}
For $p=2$ there is nothing to prove, so assume that $p \geq 3$. It is well known
(see \cite[III.15]{Huppert:I})
that $P$ is a direct product $P = \prod_{i=1}^m P_{r_i}$,
where $r_i \geq 1$ for all~$i$, and $P_r$ is the iterated wreath product
$C_p \wr C_p \wr \cdots \wr C_p$ with $r$ factors~$C_p$.
Recall that the Sylow $p$-subgroups of~$S_{p^r}$ are isomorphic to $P_r$.

As $P_1 \cong C_p$ is abelian, and $P_r$ satisfies Hypothesis~\ref{hypo:experiment} for $r \geq 2$ by Corollary~\ref{coroll:exper}, Lemma~\ref{lemma:exper2} tells us that $G = P$ satisfies Conjecture~\ref{conj:Y2}\@.

For Conjecture~\ref{conj:Y1}, observe that
$J(G_1 \times G_2) = J(G_1) \times J(G_2)$ for any groups $G_1$,~$G_2$.
Huppert shows in \cite[Satz III.15.4a)]{Huppert:I} that
$J(P_r)$ is an abelian normal subgroup of~$P_r$. Hence $J(P)$ is abelian and normal in~$P$. So $J(P) \leq \yy(P)$ by Lemma~\ref{lemma:Yportmanteau}~\partref{enum:abnormY}\@.
\end{proof}

\noindent
The fact that they provide the proper degree of generality for the proof of
\cite[Theorem 1.2]{newol} was not the only reason for introducing
$\yy(S)$ and its companion Conjecture~\ref{conj:Y2}\@. 

\begin{theorem}
\label{thm:ZY}
Let $p$ be a prime and $G$ a finite $p$-group.
Suppose that $G$ satisfies one of the following conditions:
\begin{enumerate}
\item
\label{enum:ZY1}
$G$ is generated by its abelian normal subgroups; or more generally
\item
\label{enum:ZY2}
$\yy(G) = G$; or more generally
\item
\label{enum:ZY3}
$\Omega_1(Z(\yy(G))) = \Omega_1(Z(G))$.
\end{enumerate}
Then Conjecture~\ref{conj:Y2} holds for~$G$.
\end{theorem}

\begin{proof}
Clearly \enref{ZY2}${}\Rightarrow{}$\enref{ZY3}\@. For \enref{ZY1}${}\Rightarrow{}$\enref{ZY2}, note from Lemma~\ref{lemma:Yportmanteau}\,\enref{abnormY} that every abelian normal subgroup lies in~$\yy(G)$.

Now suppose that $V$ is a faithful $F$-module.
Timmesfeld's replacement theorem in the version given in \cite[Theorem~4.1]{newol}
states that there are quadratic elements in~$G$. Hence by
Lemma~\ref{lemma:prestrange} there are quadratic elements which are \emph{late}
in the sense of~\S\ref{section:late}\@. But Lemma~\ref{lemma:strange} says that if a
quadratic element is late then it lies in $\Omega_1(Z(\yy(G)))$.
So if \enref{ZY3} holds then there are quadratic
elements in $\Omega_1(Z(G))$.
\end{proof}

\noindent
For $p \geq 5$, this result has been obtained independently by
Justin Lynd: see the remark after~\cite[Lemma~8]{Lynd:2subnormal}\@.
One application of this result is the following:

\begin{theorem}
\label{thm:GLn}
Let $n \geq 1$, let $p$~be a prime, and let $q$ be any prime power.
Let~$P$ be a Sylow $p$-subgroup of the general linear group
$\GL_n(\f[q])$. Then Conjecture~\ref{conj:Y1} holds for $S = P$,
and Conjecture~\ref{conj:Y2} holds for $G = P$.
\end{theorem}

\begin{proof}
First we consider the case where $q$ is a power of~$p$
(defining characteristic).
If we can show that $P$~is generated by its abelian normal subgroups then we
are done: for this is condition~\enref{ZY1} of Theorem~\ref{thm:ZY},
which implies both condition~\enref{ZY2} and the conclusion of that theorem.
Now, we may choose $P$~to be the group of upper triangular matrices with ones
on the diagonal. This copy of~$P$ is generated by the collection of
abelian normal subgroups $N_{i,j}$ for $1 \leq i < j \leq n$, where
\[
N_{i,j} = \{A \in P \mid \text{$A_{aa} = 1$; $A_{ab} = 0$ for $a \neq b$
whenever $a > i$ or $b < j$}\} \, .
\]
See also~\cite{Weir:GeneralLinear}\@.

Now we turn to the case where $(p,q)=1$ (coprime characteristic)\@.
It was first shown by Weir~\cite{Weir:Classical} that $P$ is a direct
product of iterated wreath products $C_{p^r} \wr C_p \wr \cdots \wr C_p$,
where the $C_{p^r}$ is the ``innermost'' factor in the wreath product.
These iterated wreath products satisfy Hypothesis~\ref{hypo:experiment} by
Lemma~\ref{lemma:exper1}\,\enref{exper1-1} and Proposition~\ref{prop:exper3}.
The one exception is of course $C_{p^r}$, which is abelian. So any direct
product of these groups satisfies Conjecture~\ref{conj:Y2} by
Lemma~\ref{lemma:exper2}\@.

We now turn to Conjecture~\ref{conj:Y1} in coprime characteristic.
As
$\yy(P)=P$ if $p=2$, we may assume that $p$~is odd. We shall show that
$J(P)$ is abelian, from which $J(P) \leq \yy(P)$ follows by
Lemma~\ref{lemma:Yportmanteau}\,\enref{abnormY}\@. Since $J(G \times H) =
J(G) \times J(H)$, we may assume that $P$ is one iterated wreath product.
Applying Lemma~\ref{lemma:Jwreath}, we see that $J(P)$ is elementary abelian
by induction on the number of iterations.
\end{proof}

\noindent
The method of Theorem~\ref{thm:ZY} can also be used to show the
following:

\begin{theorem}
\label{thm:noGenQuad}
Let $p$~be a prime, $G$ a finite $p$-group and $V$ a faithful $\f G$-module
with no
quadratic elements in $\Omega_1(Z(G))$. Then the subgroup generated by all
quadratic elements of~$G$ is a proper subgroup of~$G$.
\end{theorem}

\begin{proof}
Let $H \leq G$ be the subgroup generated by all \emph{last} quadratic elements,
c.f.\@~\S\ref{section:late}\@. By Lemma~\ref{lemma:prestrange}
we know that $H \neq 1$.
So $H \nleq Z(G)$, for $Z(G)$ has no quadratic elements.
Hence $C_G(H) \lneq G$. But $C_G(H)$ contains every quadratic element
by Lemma~\ref{lemma:lastQuadratic}\@.
\end{proof}

\subsection*{Structure of the paper}
In Section~\ref{section:Y} we
construct and study the characteristic subgroup $\yy(G)$~of $G$\@.
We then prove some orthogonality relations
in~\S\ref{section:orthog} and study weakly closed elementary abelian subgroups
in~\S\ref{section:jammed}\@. Then in Section~\ref{section:hypothesis}
we introduce a hypothesis which constitutes a strengthening of
Conjecture~\ref{conj:Y2} for groups with cyclic centre. In
Section~\ref{section:wreath} we demonstrate the key property of this
hypothesis: it remains valid when passing from a group $P$ to the wreath
product $P \wr C_p = P^p \rtimes C_p$.

In Section~\ref{section:deepest}
we introduce the notions of deepest
commutators and locally deepest commutators. In Section~\ref{section:late}
we use late and last quadratics to demonstrate some lemmas which are required
by the proofs of Theorem \ref{thm:ZY}~and \ref{thm:noGenQuad}\@.
Finally we show in~\S\ref{section:powerful} that powerful $p$-groups
satisfy Conjecture~\ref{conj:Y2}\@.

\section{A modification of Oliver's construction}
\label{section:Y}

\begin{defn}
Let $p$ be a prime, $S$ a finite $p$-group and $N \trianglelefteq S$ a normal subgroup.
A \emph{$Y$-series} in~$S$ for~$N$ is a sequence
$1 = Y_0 \leq Y_1 \leq \cdots \leq Y_n = N$ of normal subgroups $Y_i \trianglelefteq S$ such
that
\[
[\Omega_1(C_S(Y_{i-1})), Y_i ; 2] = 1
\]
holds for each $1 \leq i \leq n$. The unique largest normal subgroup~$N$ which admits
a $Y$-series is denoted $\yy(S)$.
\end{defn}

\begin{lemma}
\label{lemma:Yportmanteau}
Let $p$ be a prime and $S$ a finite $p$-group.
\begin{enumerate}
\item
If $1 = Y_0 \leq Y_1 \leq \cdots \leq Y_n = N$ is a $Y$-series in~$S$ and $M \trianglelefteq S$
also admits a $Y$-series, then there is a $Y$-series for $MN$ starting with
$Y_0,\ldots,Y_n$.
\item
There is indeed a unique largest normal subgroup $\yy(S)$ admitting a $Y$-series.
Moreover, $\yy(S)$ is characteristic in~$S$.
\item
\label{enum:abnormY}
Every abelian normal subgroup of~$S$ lies in $\yy(S)$. In particular,
$\yy(S)$ is centric in~$S$.
\item
If $p$~is odd then $\yy(S) \leq \xx(S)$.
If $p = 3$ then $\yy(S) = \xx(S)$.
\item
If $p=2$ then $\yy(S) = S$.
\end{enumerate}
\end{lemma}

\begin{proof}
\begin{enumerate}
\item
See the corresponding proof for $Q$-series, at the bottom of p.~334
in~\cite{Oliver:MartinoPriddyOdd}\@.
\item
Unique existence follows from the first part. It is characteristic because it is unique.
\item
The proof of \cite[Lemma~3.2]{Oliver:MartinoPriddyOdd} works for $\yy(S)$ as well, even
for $p=2$.
\item
If $p \geq 3$ then every $Y$-series is a $Q$-series. If $p = 3$ then the two notions coincide.
\item
If not, then $\yy(S) < S$. Choose $T > \yy(S)$ such
that $T / \yy(S)$ is cyclic of order~$2$ and contained in
$\Omega_1(Z(S/\yy(S)))$. Then $T \trianglelefteq S$, and we must
have $[\Omega_1(Z(\yy(S))),T;2] \neq 1$, since $\yy(S)$ is centric
and the $Y$-series for $\yy(S)$ cannot be extended to include~$T$.

As $\yy(S)$ has index $2$ in $T$, we have $t^2 \in \yy(S)$ for each
$t \in T$. Hence by Eqn~(2.2) of~\cite{newol} we have
\[
[\Omega_1(Z(\yy(S))),t;2] = [\Omega_1(Z(\yy(S))),t^2] = 1 \, ,
\]
a contradiction.
\qedhere
\end{enumerate}
\end{proof}

\begin{bsp}
Oliver remarks in~\cite[p.~335]{Oliver:MartinoPriddyOdd} that $\xx(S) = C_S(\Omega_1(S))$
for any finite $2$-group~$S$. So if $S$ is dihedral of order~$8$,
then $\xx(S) = Z(S)$ is cyclic of order~$2$, whereas $\yy(S)=S$ has
order~$8$. Hence $\xx(S) < J(S) = \yy(S) = S$ in this case.
\end{bsp}

\begin{bsp}
Let $S = C_5^3 \rtimes C_5$ be the semidirect product in which the cyclic group on top acts
via a $(3 \times 3)$ Jordan block (with eigenvalue~$1$)
on the rank three elementary abelian on the bottom.
Then $C_5^3 = J(S) = \yy(S) \lneq \xx(S) = S$.
\end{bsp}

\begin{rk}
Let $p$ be an odd prime.  The proof of \cite[Lemma~2.3]{oliver} also shows
that for every $p$-group~$G$ there is a $p$-group $S$ with
$S / \yy(S) \cong G$.
\end{rk}

\section{Orthogonality}
\label{section:orthog}

\begin{notn}
Suppose that $p$ is an odd prime, that $G$ is a non-trivial $p$-group,
and that $V$ is a faithful (right) $\f G$-module. Let $I$ be
the kernel of the structure map $\f G \rightarrow \End(V)$.
Recall from \cite[Section~2]{newol} that $[v,g] = v(g-1)$ for all $v \in V$, $g \in G$.
Hence $[V,g,h]=0$ if and only if $(g-1)(h-1) \in I$.

For elements $g,h \in G$ we write $g \perp_V h$ or simply $g \perp h$
if $[g,h]=1$ and $(g-1)(h-1) \in I$. Note that $\perp_V$ is a symmetric relation on~$G$.
We write\[
g^{\perp} := \{h \in G \mid g \perp h\} \, .
\]
Observe that $1 \neq g \in G$ is quadratic if and only if $g \perp g$.
\end{notn}

\begin{lemma}
\label{lemma:orthogPort}
Let $V$ be a faithful $\f G$-module, where $G$ is a non-trivial $p$-group.
For any $g,h,x \in G$ we have:
\begin{enumerate}
\item
\label{enum:orthog1}
$g$~is quadratic if and only if $g \in g^{\perp}$.
\item
\label{enum:orthog2}
The relation $h \in g^{\perp}$ is symmetric.
\item
\label{enum:orthog3}
The set $g^{\perp}$ is a subgroup of $C_G(g)$.
\item
\label{enum:orthog4}
For any integer~$r$ coprime to~$p$, we have that
$(g^r)^{\perp} = g^{\perp}$ and therefore
\[
\text{$g^r$ quadratic} \; \Longleftrightarrow \; \text{$g$ quadratic.}
\]
\item
\label{enum:orthog5}
Assume that $p$ is odd.
If $g,h$ are both quadratic and $[g,h] = 1$ then
\[
\text{$gh$ is quadratic} \; \Longleftrightarrow \; g \perp h \, .
\]
\item
\label{enum:orthog6}
Assume that $p$ is odd.
If $g$ is quadratic and $[g,g^x] = 1$ then
\[
g \perp g^x \; \Longleftrightarrow \; \text{$[g,x]$ is quadratic.}
\]
\end{enumerate}
\end{lemma}

\begin{proof}
Parts \enref{orthog1}~and \enref{orthog2} are clear.
Part~\enref{orthog3}:
Obviously $g \perp 1$. If $g \perp h$ and $g \perp k$, then
\[
(g-1)(hk-1) = (g-1)(h-1) + h(g-1)(k-1)
\]
and so $g \perp hk$. Inverses follow, since $G$ is finite.

Part~\enref{orthog4}: If $g=1$ then there is nothing to prove. If $g \neq 1$ then $r$ is
a unit modulo the order of~$g$. So it suffices to show the inclusion $g^{\perp} \leq (g^r)^{\perp}$.
As $g^r$ only depends on the residue class of~$r$ modulo the order of~$g$, we may assume that
$r \geq 1$. Then
\[
g^r - 1 = a (g-1) \quad \text{for} \; a = \sum_{i=0}^{r-1} g^i \, .
\]
So if $(g-1)(h-1)$ lies in the kernel~$I$ of the representation,
then so does~$(g^r-1)(h-1)$.

Part~\enref{orthog5}:
From $(gh-1) = (g-1) + g(h-1)$, we get that
\[
(gh-1)^2 = (g-1)^2 + 2g(g-1)(h-1) + g^2(h-1)^2
\]
because $[g,h]=1$. So since $2g$~is invertible and $(g-1)^2$, $(h-1)^2 \in I$,
we see that $(gh-1)^2 \in I$ if and only if $(g-1)(h-1) \in I$.

Part~\enref{orthog6}:
By~\enref{orthog4} we see that
$g^{-1}$~is quadratic if and only if~$g$ is; and
by~\enref{orthog3} we have
$g^{-1} \in (g^x)^{\perp}$ if and only if $g \in (g^x)^{\perp}$.
But $[g,x] = g^{-1} g^x$. Now apply~\enref{orthog5}\@.
\end{proof}

\begin{corollary}
\label{coroll:quadOrthog}
Let $p$ be an odd prime, $G$ a non-trivial $p$-group, and $V$ a faithful $\f G$-module.
If $E = \langle g_1,g_2 ,\ldots, g_r \rangle \leq G$ is elementary abelian then
the following three statements are equivalent:
\begin{enumerate}
\item
Every element $1 \neq g \in E$ is quadratic.
\item
$g \perp h$ for all $g,h \in E$.
\item
$g_i \perp g_j$ for all $i,j \in \{1,\ldots,r\}$.
\end{enumerate}
\end{corollary}

\begin{notn}
Recall that an elementary abelian subgroup $E \leq G$ is called \emph{quadratic}
(for~$V$) if it satisfies the equivalent conditions of Corollary~\ref{coroll:quadOrthog}\@.
\end{notn}

\begin{proof}
Clearly the second statement implies the first. The first implies the third, by
Lemma~\ref{lemma:orthogPort}\,\enref{orthog5}\@. Now assume the third condition holds. Then
$E \leq (g_i)^{\perp}$ for every~$i$, since $(g_i)^{\perp}$ is a group
(Lemma~\ref{lemma:orthogPort}\,\enref{orthog3})
and contains each~$g_j$. So $g \perp g_i$ for each~$i$ and for each $g \in E$.
Hence $g_i \in g^{\perp}$ for every~$i$. Therefore $E \leq g^{\perp}$. So $g \perp h$ for all
$g,h \in E$, and the second condition holds.
\end{proof}

\begin{lemma}
\label{lemma:orthogCent}
Suppose that $g,h \in G$ with $g \neq 1$ and $g \perp h$. Suppose further that $C$~is a subgroup
of~$C_G(h)$ which contains~$g$ and has cyclic centre. Then
\[
\Omega_1(Z(C)) \leq h^{\perp} \, .
\]
\end{lemma}

\begin{proof}
We proceed by induction on the smallest integer $r \geq 1$ with
$g \in Z_r(C)$. If $r=1$ then $\Omega_1(Z(C)) \leq
\langle g \rangle$, and we are done since $\langle g \rangle \leq h^{\perp}$.
If $g \in Z_{r+1}(C)$, then there is an $x \in C$ such that
$1 \neq y := [g,x] \in Z_r(C)$.
By induction it suffices to show that $y \perp h$. Since $x \in C$ we have
$[h,x]=1$ and therefore $x^{-1} (g-1)(h-1)x = (g^x - 1)(h-1)$,
showing that $g^x \perp h$. So $y \in \langle g,g^x \rangle \leq h^{\perp}$.
\end{proof}

\noindent
We close this section by recalling without proof
a key lemma from~\cite{oliver}\@.

\begin{lemma}[Lemma 4.1 of~\protect\cite{oliver}]
\label{lemma:GHL}
Suppose that $p$ is an odd prime, that $G$ is a non-trivial $p$-group,
and that $V$ is a faithful (right) $\f G$-module.
Suppose that $A,B \in G$ are such that $C := [B, A]$ is a nontrivial
element of $C_G(A, B)$. If $B$ is quadratic, then so is~$C$.
\qed
\end{lemma}

\section{Weakly closed subgroups}
\label{section:jammed}
\noindent
This section is related to work of Chermak and
Delgado~\cite{ChermakDelgado:Measuring},
especially the case $\alpha = 1$ of their Theorem~2.4\@.

\begin{notn}
For the sake of brevity we will say that an abelian subgroup $A$~of $G$
is \emph{weakly closed} if $A$~is weakly closed in $C_G(A)$ with
respect to~$G$. 
That is, $A \leq G$ is weakly closed if
$[A,A^g] \neq 1$ holds for every $G$-conjugate
$A^g \neq A$.
\end{notn}

\begin{rk}
Every maximal elementary abelian subgroup $M$~of $G$ is weakly closed,
since if $M^g \neq M$ but $[M,M^g]=1$ then $\langle M,M^g \rangle$ is
elementary abelian and strictly larger. Hence every elementary abelian
subgroup is contained in a weakly closed one.
If the normal closure of~$E$ is non-abelian, then every weakly closed
elementary abelian subgroup containing~$E$ is non-normal.
\end{rk}

\begin{remark}
\label{rk:Z2NG}
Note that if $A \leq G$ is an abelian subgroup and $g \in G$ then
\[
[A, A^g] = 1 \; \Longleftrightarrow \; [A, [A,g]] = 1 \; \Longleftrightarrow \;
[g,A,A] = 1 \, .
\]
In particular, if $A$ is weakly closed (in our sense),
then $Z_2(G) \leq N_G(A)$: for if $g \in Z_2(G)$ then $[A,g] \leq Z(G)$ and therefore
$[A,[A,g]] = 1$. Hence $[A,A^g] = 1$ and so $A = A^g$, that is $g \in N_G(A)$.
\end{remark}

\begin{remark}
Using GAP~\cite{GAP4}, the authors have constructed the following examples:
\begin{itemize}
\item
The Sylow $3$-subgroup $G$ of the symmetric group $S_{27}$ contains a rank four
weakly closed elementary abelian~$E$ with $E \cap Z(G) = 1$.
\item
The Sylow $3$-subgroup $G$ of the symmetric group $S_{81}$ contains a rank six
weakly closed elementary abelian~$E$ with $E \cap Z_2(G) = 1$.
\end{itemize}
The GAP code is available from the first author on request.
\end{remark}

\begin{lemma}
\label{lemma:jammedNorm}
Suppose that $G$ is a finite $p$-group and that the elementary abelian
subgroup $E$~of $G$ is weakly closed. Then $N_G(E) = N_G(C_G(E))$.
So if $E$ is not central in~$G$ then $N_G(E) \gneq C_G(E)$.
\end{lemma}

\begin{proof}
$N_G(E)$ always normalizes~$C_G(E)$. If $x \in G$ normalizes $C_G(E)$ then
as $E \leq C_G(E)$ we have that $E^x \leq C_G(E)$ and therefore $[E,E^x]=1$.
So $E^x=E$, for $E$ is weakly closed. Hence $x \in N_G(E)$. Last part: $G$ is a nilpotent
group. So if $C_G(E)$ is a proper subgroup of~$G$, then it is properly contained in its normalizer.
\end{proof}

\begin{notn}
Let $G$ be a finite group, $H \leq G$ a subgroup, and $V$ a faithful $\f G$-module.
Following Meierfrankenfeld and Stellmacher~\cite{MeierfStellm:OtherPGV} we set
\[
j_H(V) := \frac{\abs{H} \cdot \abs{C_V(H)}}{\abs{V}} \in \mathbb{Q} \, .
\]
This means that an elementary abelian subgroup $E$~of $G$ is an offender
if and only if $E\neq 1$ and $j_E(V) \geq 1 = j_1(V)$.
\end{notn}

\begin{lemma}[Lemma 2.6 of \protect\cite{MeierfStellm:OtherPGV}]
\label{lemma:MfS}
Let $G$ be a finite group and $V$ a faithful $\f G$-module. Let $H,K$ be subgroups
of~$G$ with $\langle H, K \rangle = HK$. Then
\[
j_{HK}(V) j_{H \cap K}(V) \geq j_H(V) j_K(V) \, ,
\]
with equality if and only if $C_V(H \cap K) = C_V(H) + C_V(K)$.
\qed
\end{lemma}

\begin{rk}
As some readers may find the
article~\cite{MeierfStellm:OtherPGV} by
Meierfrankenfeld and Stellmacher
hard to obtain, we reproduced the proof of this result in
our earlier paper~\cite[Lemma~3.1]{newol} -- though unfortunately we accidentally omitted the
necessary assumption that $\langle H,K \rangle = HK$.
\end{rk}

\noindent
Recall that a faithful $\f G$-module $V$ is called an \emph{$F$-module} if it
has at least one offender.

\begin{proposition}
\label{prop:jammed}
Suppose that the faithful $\f G$-module $V$ is an $F$-module. Set
\[
j_0 = \max \{ j_E(V) \mid \text{$E$~an offender} \} \, .
\]
Then there is a weakly closed quadratic offender~$E$ with $j_E(V)=j_0$.

Moreover if $D \leq G$ is any offender with $j_D(V)=j_0$,
then there is such an~$E$ which is a subgroup of
the normal closure of $D$.
\end{proposition}

\begin{proof}
Let $D$ be an offender with $j_D(V) = j_0$.
Then $D$~has a subgroup $C \leq D$ which is minimal by inclusion amongst the
offenders with $j_C(V) = j_0$.
By maximality of~$j_0$, the version of Timmesfeld's replacement theorem in
\cite[Theorem 4.1]{newol} then tells us that~$C$ is a quadratic offender.

Suppose first that $j_0 > 1$. We shall show that $C$~is weakly closed.
If not, then $A := \langle C, C^g \rangle$ is
elementary abelian for some $g \in G$ with $C^g \neq C$.
Then $j_A(V) \leq j_0$ by maximality of~$j_0$. And
since $j_0>1 = j_1(V)$, we have that
$j_{C \cap C^g}(V) < j_0$ by maximality of~$j_0$
and minimality of~$C$.
But by Lemma~\ref{lemma:MfS},
this means that
\[
j_0^2 > j_A(V) j_{C \cap C^g}(V) \geq j_C(V) j_{C^g}(V) = j_0^2
\, ,
\]
a contradiction. So $E = C \leq D$ has the required properties.

Now suppose that $j_0 = 1$.
Let $g \in G$ be such that $C^g \neq C$ and $[C,C^g]=1$.
As $\langle C,C^g \rangle$ and $C \cap C^g$ both
have $j \leq 1$, Lemma~\ref{lemma:MfS} means that
both have $j = 1$,
and (by equality) $C_V(C \cap C^g) = C_V(C) + C_V(C^g)$. Furthermore,
minimality of~$C$ means that $C \cap C^g = 1$, and so
$C_V(C) + C_V(C^g) = V$.
As in the proof of \cite[Lemma 4.3]{newol}, this means that $[V,C,C^g]=0$.
Let $T \subseteq G$ be a subset maximal
with respect to the condition that $E := \langle C^g \mid g \in T\rangle$ is
abelian (and therefore elementary abelian)\@.
Applying the above argument to any $g,h \in T$ with $C^g \neq C^h$ we deduce
that $[V,C^g,C^h] = 0$. So since each $C^g$ is quadratic, we deduce that~$E$
is quadratic too. Finally, a repeated application of Lemma~\ref{lemma:MfS}
coupled with the fact that $j$ never
exceeds~$1$ tells us that $j_E(V)=1$ too.
But by construction, $E$ is weakly closed and contained in the normal
closure of~$D$.
\end{proof}

\ignore{
\begin{lemma}
\label{lemma:rankJammedQuadratic}
Suppose that $p$ is odd, and that $(G,V)$ is a counterexample to Conjecture~\ref{conj:GHL}\@.
Then every weakly closed quadratic offender has rank at least~$p$.
\end{lemma}

\begin{proof}
By \cite[Lemma~8.2]{newol}, any offender~$E$ must have rank at least $p-1$;
and part~(2) of that lemma says that if $E$ has rank $p-1$ then it is
quadratic; its normal closure $D$~in $N_G(N_G(E))$ is elementary abelian;
and $D$ strictly contains $E$. So $E$ cannot be weakly closed.
\end{proof}
} 

\noindent
It follows from Proposition~\ref{prop:jammed} that Conjecture~\ref{conj:Y2}
holds for $(G,V)$ if the following conjecture does.

\begin{conjecture}
\label{conj:weakClosure}
Let $p$~be a prime, $G$ a finite $p$-group, and $V$ a faithful
$\f G$-module. If there is an elementary abelian subgroup
$1 \neq E \leq G$ which is both quadratic on~$V$ and weakly closed in
$C_G(E)$ with respect to~$G$,
then there are quadratic elements in $\Omega_1(Z(G))$.
\end{conjecture}

\noindent
We establish Theorem~\ref{thm:ol3main} by demonstrating that the Sylow
subgroups of the symmetric groups satisfy Conjecture~\ref{conj:weakClosure}\@.

\section{An inductive hypothesis}
\label{section:hypothesis}

\noindent
We now present the inductive hypothesis (Hypothesis~\ref{hypo:experiment}) that will be used
in Section~\ref{section:wreath} to verify Conjecture~\ref{conj:Y2} for an iterated wreath product.
First though, we need a few auxilliary lemmas.

\begin{lemma}
\label{lemma:preExp}
Suppose that $G$ is a direct product of the form $G = H \times P$,
where $H$~and $P$ are $p$-groups. Let $E \leq G$ be an elementary abelian
subgroup which is weakly closed (in $C_G(E)$ with respect to~$G$)\@.
Set
\[
F = \{g \in P \mid \exists h \in H \; (h,g) \in E\} \leq P \, ,
\]
and set $N = N_P(F)$. 
Then the following hold:
\begin{enumerate}
\item
\label{enum:preExp-1}
$F$ is weakly closed (in $C_P(F)$ with respect to~$P$).
\item
\label{enum:preExp-2}
$1 \times [F,N] \leq E$.
\item
\label{enum:preExp-3}
If $E \nleq H \times Z(P)$, then $[F,N] \neq 1$ and therefore $E \cap (1 \times P) \neq 1$.
\end{enumerate}
\end{lemma}

\begin{proof}
\begin{enumerate}
\item
If $x \in P$ and $F^x \neq F$ then $E^{(1,x)} \neq E$ and therefore
$[E,E^{(1,x)}] \neq 1$. But $[E,E^{(1,x)}] = 1 \times [F,F^x]$. So
$[F,F^x] \neq 1$.
\item
Let $f \in F$ and $n \in N$. Pick $h \in H$ such that $(h,f) \in E$. Since
$[F,F^n] = [F,F]=1$, we have $[E,E^{(1,n)}] = 1$ by the proof of the first part.
Hence $E^{(1,n)} = E$, and therefore $(1,[f,n]) = [(h,f),(1,n)] \in E$.
\item
$F$ is weakly closed by the first part, and non-central by assumption. 
Hence $1 \neq [F,N] \leq F$ by Lemma~\ref{lemma:jammedNorm}\@. Done by the second part.
\qedhere
\end{enumerate}
\end{proof}

\begin{hypothesis}
\label{hypo:experiment}
Let $P$ be a $p$-group with the following properties:
\begin{itemize}
\item
$P$ is nonabelian with cyclic centre.
\item
Suppose that $H$ is a $p$-group and that $V$ is a faithful $\f G$-module,
where $G = H \times P$.
If $E$ is a weakly closed quadratic elementary abelian subgroup
of~$G$ which satisfies $E \nleq H \times Z(P)$,
then the subgroup $1 \times \Omega_1(Z(P))$ of $Z(G)$ is quadratic.
\end{itemize}
\end{hypothesis}

\noindent
At the end of this section (Corollary~\ref{coroll:exper1}) we give some
first examples of groups that satisfy this hypothesis. In the next section
(Corollary~\ref{coroll:exper}) we shall show that the Sylow $p$-subgroups
of $S_{p^n}$ also satisfy it.
First though we explain the significance of the hypothesis for
Oliver's conjecture.

\begin{lemma}
\label{lemma:exper2}
Suppose that the finite $p$-group $G$ is a direct product
$G = \prod_{r=1}^n H_r$, where each $H_r$ is either abelian or
satisfies Hypothesis~\ref{hypo:experiment}\@.
Then $G$ satisfies Conjectures \ref{conj:Y2}~and \ref{conj:weakClosure}\@.
\end{lemma}

\begin{proof}
For $1 \leq r \leq n$, we shall denote
by~$K_r$ the product $\prod_{i \neq r} H_i$. Hence $G = K_r \times H_r$ for each~$r$.
Note that $Z(G) = \prod_r Z(H_r)$.
\par
Conjecture~\ref{conj:weakClosure} implies Conjecture~\ref{conj:Y2}
by Proposition~\ref{prop:jammed}\@, and so we assume that there
is a weakly closed quadratic elementary abelian subgroup $E$~of $G$ with $E \neq 1$.
If $E \leq Z(G)$ then every element $1 \neq g \in E$ is a quadratic element of $\Omega_1(Z(G))$.
If $E \nleq Z(G) = \prod_r Z(H_r)$ then for some $1 \leq r \leq n$ we have
$E \nleq K_r \times Z(H_r)$.
It follows that $H_r$ cannot be abelian, and so by assumption it satisfies
Hypothesis~\ref{hypo:experiment}: which means that the subgroup
$1 \times \Omega_1(Z(H_r))$ of $\Omega_1(Z(G))$ is quadratic.
\end{proof}

\noindent
We are particularly interested in wreath products of the form
$G = P \wr C_p$, where $P$ is a $p$-group.
Recall that this means that $G$~is the semidirect product
$G = P^p \rtimes C_p$, where the $C_p$ on top acts on the base
$P^p = \prod_1^p P$ by permuting the factors cyclically. In particular,
this means that we may view $P^p$ as a subgroup of~$G$.
\par
We start with a minor diversion.
The following result generalizes \cite[Satz III.15.4a)]{Huppert:I}
slightly and is presumably well known. It is needed for the proof of
Theorem~\ref{thm:GLn}\@.

\begin{lemma}
\label{lemma:Jwreath}
Suppose that $p$ is an odd prime and
that $P$~is a $p$-group such that $J(P)$ is elementary abelian.
Then $J(P \wr C_p)$ is elementary abelian too. In particular,
if $P \neq 1$ then
$J(P \wr C_p)$ is the copy of $J(P)^p$ in the base subgroup
$P^p \leq P \wr C_p$.
\end{lemma}

\begin{proof}
If $P = 1$ then $P \wr C_p = C_p$ is elementary abelian and we are done.
So we may assume that $P \neq 1$, which means that $r \geq 1$, where
$r$~is the $p$-rank of~$P$. Note that $J(P) \cong C_p^r$ by assumption.
\par
Consider the base subgroup $P^p$, which has index~$p$ in
$P \wr C_p = P^p \rtimes C_p$. Since $J(G_1 \times G_2) = J(G_1) \times J(G_2)$,
we see that $J(P^p) = J(P)^p$, which is elementary abelian of rank~$pr$.
So if $J(P \wr C_p)$ is not elementary abelian then $P^p \rtimes C_p$
must contain an elementary abelian subgroup $E$ of rank${}\geq pr$ with
$E \nleq P^p$. We set $E' = E \cap P^p$. Since $P^p$ has index~$p$ in
$P \wr C_p$, it follows that $\abs{E\colon E'} = p$ and therefore that
$E' \leq P^p$ is elementary abelian of rank${}\geq pr-1$.
Furthermore there is an element $g \in (P \wr C_p) \setminus P^p$
with $E = \langle E', g \rangle$ and therefore $[E',g] = 1$.
\par
We split $P^p$ as $P^p = A \times B$, where
$A,B \leq P^p$ are given by
\begin{xalignat*}{2}
A & = \{(a_1,1,\ldots,1) \mid a_1 \in P \} &
B & = \{(1,a_2,\ldots,a_p) \mid a_i \in P \} \, .
\end{xalignat*}
Note that $E' \cap B$ must be trivial: for
no non-trivial element of $B$ can commute with
$g \in (P \wr C_p) \setminus P^p$. Hence the projection of $E'$ onto~$A$ is
injective. As $A \cong P$ has rank~$r$, this means that $pr-1 \leq r$. This is
impossible with $p$~odd and $r \geq 1$.
\end{proof}

\noindent
Now we derive a property of wreath products which will
be useful for verifying the hypothesis.

\begin{lemma}
\label{lemma:preExp2}
Suppose that $G$ is the wreath product $P \wr C_p$, where $P$ is a finite $p$-group. Suppose that
$E$ is a weakly closed elementary abelian subgroup of~$G$ and that
either of the following properties is
satisfied:
\begin{enumerate}
\item
\label{enum:preExp2-1}
$E$ is not contained in the base subgroup $P^p$~of $G$.
\item
\label{enum:preExp2-2}
$Z(P)$ is cyclic, $E$ is non-central, and $E$ lies in an abelian normal subgroup of~$G$.
\end{enumerate}
Then $\Omega_1(Z(G)) \leq [E,N_G(E)]$.
\end{lemma}

\begin{proof}
The group $\Omega_1 (Z(G))$ is cyclic of order~$p$. It is generated by
$(z,\ldots,z) \in P^p$ for any element
$z \in \Omega_1(Z(P))$ with $z \neq 1$.
\par
Suppose first that~\enref{preExp2-2} holds.
As $E$ lies in an abelian normal subgroup of~$G$ and is weakly closed,
we deduce
that $E \trianglelefteq G$. So $[E,N_G(E)] = [E,G]$ is normal too. Moreover,
$[E,G]$ is nontrivial, because $E$ is non-central.
Since $\abs{\Omega_1(Z(G))}=p$, it follows that $\Omega_1(Z(G)) \leq [E,G] = [E,N_G(E)]$.
\par
Now suppose that~\enref{preExp2-1} holds.
Pick an element $z \in \Omega_1(Z(P))$ with $z \neq 1$.
It suffices to show that $(z,\ldots,z) \in [E,N_G(E)]$.
We start by showing that
$c = (1,z,z^2,\ldots,z^{p-1})$ satisfies $c \in Z_2(G)$.
\par
If $g \in P^p$ then $[g,c]=1$, since $c \in Z(P^p)$.
If $g \in G \setminus P^p$ then $g = h\sigma$, where
$h \in P^p$ and $\sigma$ is a cyclic permutation of the $p$ factors of~$P^p$\@.
Hence
\[
[g,c] =[\sigma,c] = (c^{\sigma})^{-1} c = (z,z,\ldots,z)^r \in Z(G) \quad
\text{for some $1 \leq r \leq p-1$.}
\]
Therefore $c \in Z_2(G)$ as claimed.
By Remark~\ref{rk:Z2NG} if follows that $c \in N_G(E)$\@.
\par
Now by assumption $E \nleq P^p$, and so
$E$~contains some $g =h\sigma$ in $G \setminus P^p$. So replacing $g$ by
a suitable power if necessary, we may suppose that $r=1$, and so
$(z,\ldots,z) = [g,c] \in [E,c] \leq [E,N_G(E)]$.
\end{proof}

\noindent
The following lemma gives one way of verifying that the hypothesis
is satisfied.

\begin{lemma}
\label{lemma:exper1}
Let $P$ be a non-abelian $p$-group with cyclic centre.
Suppose that for every non-central
weakly closed elementary abelian subgroup $F$~of $P$ we have that
$[F,N_P(F)] \cap Z_2(P) \neq 1$.
Then $P$ satisfies Hypothesis~\ref{hypo:experiment}\@.
\end{lemma}

\begin{proof}
Let $G = H \times P$ for a $p$-group~$H$, and let $E$ be a weakly closed
quadratic elementary abelian subgroup of~$G$ with $E \nleq H \times Z(P)$.
We need to show that $1 \times \Omega_1(Z(P))$ is quadratic.
\par
Consider the elementary abelian subgroup $F \leq P$ defined by
$F = \{g \in P \mid \exists h \in H \; (h,g) \in E\}$.
Then $F$ is weakly closed (in $C_P(F)$ with respect to~$P$) by
Lemma~\ref{lemma:preExp}\,\enref{preExp-1}\@.
Moreover, $F \nleq Z(P)$, since $[F,N_P(F)] \neq 1$ by Lemma~\ref{lemma:preExp}\,\enref{preExp-3}\@.
So by assumption we may pick
$g \in Z_2(P) \cap [F,N_P(F)]$ with $g \neq 1$.
Since $g \in [F,N_P(F)]$, we have $(1,g) \in E$
by Lemma~\ref{lemma:preExp}\,\enref{preExp-2}\@. Hence $(1,g)$ is
quadratic and of order~$p$. If $g \in Z(P)$ then we are done,
since $Z(P)$ is cyclic.
\par
If $g \in Z_2(P) \setminus Z(P)$ then
$1 \neq [g,x]$ generates $\Omega_1(Z(P))$ for some $x \in P$. But
then $(1,[g,x]) = [(1,g),(1,x)]$ is quadratic by Lemma~\ref{lemma:GHL}\@.
\end{proof}

\begin{corollary}
\label{coroll:exper1}
Let $p$ be a prime and $P$ a finite $p$-group. Assume that either of the following holds:
\begin{enumerate}
\item
\label{enum:exper1-1}
$P$~is a wreath product of the form $P = C_{p^r} \wr C_p$ for $r \geq 1$;
\item
\label{enum:exper1-2}
$Z(P)$ is cyclic, and $P$ has nilpotence class two or three.
\end{enumerate}
Then Hypothesis~\ref{hypo:experiment} holds for $P$\@.
\end{corollary}

\begin{proof}
In both cases, $P$ is nonabelian and has cyclic centre.
Let $F \leq P$ be a non-central elementary abelian subgroup which is weakly closed
in $C_P(F)$ with respect to~$P$. By Lemma~\ref{lemma:exper1} it suffices to show that
$Z_2(P) \cap [F,N_P(F)] \neq 1$.
\begin{enumerate}
\item
We prove that $\Omega_1(Z(P)) \leq [F,N_P(F)]$.
The case $F \nleq C_{p^r}^p$ is
covered by Lemma~\ref{lemma:preExp2}\,\enref{preExp2-1}\@. Since
$C_{p^r}^p$ is normal abelian, the case $F \leq C_{p^r}^p$ is accounted for
by Lemma~\ref{lemma:preExp2}\,\enref{preExp2-2}\@.
\item
Since $F$ is weakly closed and non-central we have $C_P(F) \lneq N_P(F)$ by
Lemma~\ref{lemma:jammedNorm}\@. Hence $[F,N_P(F)] \neq 1$.
And as the (nilpotence) class of~$P$ is at most three we have
$[F,N_P(F),P] \leq Z(P)$,
whence $[F,N_P(F)] \leq Z_2(P)$.
\qedhere
\end{enumerate}
\end{proof}

\section{Wreath products}
\label{section:wreath}
\noindent
We saw in Lemma~\ref{lemma:exper2} that Conjecture~\ref{conj:Y2} holds for the direct
product $H_1 \times \cdots \times H_n$ if each $p$-group $H_r$ is abelian or
satisfies Hypothesis~\ref{hypo:experiment}\@. In Corollary~\ref{coroll:exper1} we saw
our first examples of groups which satisfy the hypothesis. We now demonstrate
that if $P$ satisfies the hypothesis, then so does the wreath
product $P \wr C_p$. This is the key step in proving that the Sylow subgroups of a symmetric
group satisfy Conjecture~\ref{conj:Y2}\@.

\begin{proposition}
\label{prop:exper3}
If the $p$-group $P$ satisfies Hypothesis~\ref{hypo:experiment} then so does the wreath
product $Q =P \wr C_p$.
\end{proposition}

\begin{proof}
Certainly $Q$ is nonabelian with cyclic centre. Let $H$ be a $p$-group, and $V$ a
faithful module for $G = H \times Q$. Recall that $Q$ is a semidirect product,
with base subgroup $K := \prod_1^p P$ on which $C_p$ acts by permuting the factors cyclically.
Let $E \leq G$ be a quadratic elementary abelian subgroup which is
weakly closed in $C_G(E)$ with respect to~$G$, and also satisfies $E \nleq H \times Z(Q)$.
We have to show that $1 \times \Omega_1(Z(Q))$ is quadratic.

Let $F = \{g \in Q \mid \exists h \in H \; (h,g) \in E\}$.
Then $F$~is a non-central elementary abelian subgroup of~$Q$;
and by Lemma~\ref{lemma:preExp} it is weakly closed
(in $C_Q(F)$ with respect to~$Q$).

Suppose first that $E \nleq H \times K$.  Then $F \nleq K$,
and so by Lemma~\ref{lemma:preExp2}\,\enref{preExp2-1} we have
$\Omega_1(Z(Q)) \leq [F, N_Q(F)]$.
The same thing happens if $E \leq H \times Z(K)$: for then $F \leq Z(K)$,
which is an abelian normal subgroup of~$Q$.
Hence $\Omega_1(Z(Q)) \leq [F, N_Q(F)]$
by Lemma~\ref{lemma:preExp2}\,\enref{preExp2-2}\@.
In both cases we then deduce
from Lemma~\ref{lemma:preExp}\,\enref{preExp-2}
that $1 \times \Omega_1(Z(Q))$ lies in~$E$
and is therefore quadratic.
\medskip

\noindent
So from now on we may assume that $E \leq H \times K$ but $E \nleq H \times Z(K)$.
That is,
$F$ is a subgroup of $K = \prod_1^p P$, but
$F \nleq Z(K) = Z(P)^p$. Define a subset $T(F) \subseteq \{1,2,\ldots,p\}$ by
\[
i \in T(F) \; \Longleftrightarrow \; F \nleq
P \times P \times \cdots \times Z(P) \times \cdots \times P \, ,
\]
where the factor $Z(P)$ occurs in the $i$th copy of~$P$\@.
Since $F \nleq Z(P)^p$, we have $T(F) \neq \emptyset$.

For $1 \leq i \leq p$ we define subgroups $L_i, P(i) \leq H \times K$ as follows:
\begin{xalignat*}{2}
L_i & = H \times (P \times P \times \cdots \times 1 \times \cdots \times P) &
P(i) & = 1 \times (1 \times 1 \times \cdots \times P \times \cdots \times 1) \, ,
\end{xalignat*}
where the brackets enclose~$K$ and the one distinguished factor occurs in the $i$th factor
of $K = P^p$. Then $H \times K = L_i \times P(i)$ for each~$i$.
Observe that
\begin{equation}
\mytag{eqn:TFnew}
i \in T(F) \; \Longleftrightarrow \; E \nleq L_i \times Z(P(i)) \, .
\end{equation}
Pick $1 \neq \zeta \in \Omega_1(Z(Q))$. We need to show that $(1,\zeta) \in Z(G)$ is quadratic.
Note that $\zeta =(z,z,\ldots,z) \in K$ for some $1 \neq z \in \Omega_1(Z(P))$.
That is,
$\zeta = z_1 z_2 \cdots z_p$, where for
$1 \leq i \leq p$ we define $z_i \in Z(K)$ by
\[
z_i =(1,\ldots,z,\ldots,1) \quad \text{($z$ at the $i$th place).}
\]
Now, the group $P(i) \cong P$ satisfies Hypothesis~\ref{hypo:experiment} by assumption.
Applying the hypothesis to the direct product group $L_i \times P(i) = H \times K$,
we see from Eqn~\eqref{eqn:TFnew} that $(1,z_i)$ is quadratic for every $i \in T(F)$.
As $T(F)$ is nonempty and all
the $z_i$ are conjugate in~$Q$ under the cyclic permutation of the factors of $K =P^p$,
it follows that \emph{every} $(1,z_i)$ is quadratic.

From Lemma~\ref{lemma:preExp}\,\enref{preExp-3} and Eqn~\eqref{eqn:TFnew}
it follows that for every $i \in T(F)$ there is some $1 \neq g_i \in P$ such that
$(1,k_i) \in E$ for $k_i = (1,\ldots,g_i,\ldots,1) \in K$.
Now suppose that $i,j \in T(F)$ are distinct. Then $(1,k_i) \perp (1,k_j)$,
as they both lie in~$E$. Applying Lemma~\ref{lemma:orthogCent} with
$C = P(i)$, we have that $(1,z_i) \perp (1,k_j)$.
Applying this lemma a second time, but now with $C = P(j)$,
we deduce that
\begin{equation}
\mytag{eqn:zperpzNew}
(1,z_i) \perp (1,z_j)
\end{equation}
holds for all $i,j \in T(F)$. We have already seen that it always holds for $i=j$.

As $Q$~is the wreath product $P \wr C_p$, we may pick an element $x \in Q \setminus K$ such that
$(h_1,\ldots,h_p)^x = (h_p,h_1,\ldots,h_{p-1})$ holds for all $(h_1,\ldots,h_p) \in P^p = K$.
Then $T(F^x) = \sigma(T(F))$, where $\sigma$~is the $p$-cycle
$(1\,2\,\ldots\,p)$. Hence $T(F^{x^m}) = \sigma^m(T(F))$.
We claim that
\begin{equation}
\mytag{eqn:sigmaTF}
\forall\,0 \leq m \leq p-1 \qquad
\sigma^m(T(F)) \cap T(F) \neq \emptyset \, .
\end{equation}
For if $\sigma^m(T(F)) \cap T(F) = \emptyset$,
then $[F^{x^m},F] = 1$ by definition of~$T(F)$.
But $F$ is weakly closed, and so $[F^{x^m},F] = 1$ implies that $F^{x^m} = F$. But then
$\sigma^m(T(F)) = T(F^{x^m}) = T(F)$, and so
$\sigma^m(T(F)) \cap T(F) = T(F) \neq \emptyset$, a contradiction.

So Eqn~\eqref{eqn:sigmaTF} is proved. We may rephrase it as follows:
\begin{equation}
\mytag{eqn:overlapNew}
\forall \, 0 \leq m \leq p-1 \quad \exists \, a,b \in T(F)
\quad b \equiv a + m \pmod{p} \, .
\end{equation}
Now pick any $i,j \in \{1,2,\ldots,p\}$.
By Eqn~\eqref{eqn:overlapNew} there are $a,b \in T(F)$ with
$b - a \equiv j - i \pmod{p}$.
Define $r$~to be the integer $0 \leq r < p$ such that $r \equiv i - a \pmod{p}$.
Then
\begin{xalignat*}{3}
i & = a + r \pmod{p} & \text{and} & & j & = b + r \pmod{p} \, .
\end{xalignat*}
Hence
\begin{xalignat*}{3}
(1, z_i) & = (1,z_a)^{(1,x^r)} & \text{and} & & 
(1, z_j) & = (1,z_b)^{(1,x^r)} \, .
\end{xalignat*}
But $(1,z_a) \perp (1,z_b)$, as this is one of the cases for which we have already demonstrated
Eqn~\eqref{eqn:zperpzNew}\@. So we deduce that $(1,z_i) \perp (1,z_j)$: that is,
Eqn~\eqref{eqn:zperpzNew} holds for all $i,j$. From $\zeta = z_1 z_2 \cdots z_p$ we therefore
deduce by Corollary~\ref{coroll:quadOrthog} that $(1,\zeta) \perp (1,\zeta)$,
that is $(1,\zeta)$ is quadratic, as required.
\end{proof}

\begin{corollary}
\label{coroll:exper}
The Sylow $p$-subgroups of the symmetric group $S_{p^n}$ satisfy Hypothesis~\ref{hypo:experiment}
for every $n \geq 2$.
\end{corollary}

\begin{proof}
By induction on~$n$. It is well known that the Sylow
$p$-subgroups $P_n$ of $S_{p^n}$ are isomorphic to the iterated wreath product
$C_p \wr C_p \wr \cdots \wr C_p$, with $n$ copies of $C_p$.
So $P_2 = C_p \wr C_p$ satisfies the hypothesis by
Corollary~\ref{coroll:exper1}\,\enref{exper1-1}\@.
If $P_r$ satisfies the hypothesis, then so does
$P_{r+1} = P_r \wr C_p$ by Proposition~\ref{prop:exper3}\@.
\end{proof}

\section{Deepest commutators}
\label{section:deepest}
\noindent
In this and the subsequent section we assemble the tools needed for the proof
of Theorems \ref{thm:ZY}~and \ref{thm:noGenQuad}\@.

\begin{notn}
Let $G$ be a $p$-group and $g \in G$ a non-central element.
Let $r_0 = \max \{r \mid \exists k \in K_r(G) \colon [g,k] \neq 1 \}$.
This exists by nilpotence, since $g$~is not central.
\par
If $k \in K_{r_0}(G)$ and $[g,k] \neq 1$, then we call $[g,k]$ a
\emph{deepest commutator} of~$g$.
\end{notn}

\begin{rk}
Observe that $r_0(g)=r_0(g^x)$ for every $x \in G$.
\end{rk}

\begin{lemma}
\label{ol3:deepestCommutator}
Let $G$ be a $p$-group and $g \in G$ a non-central element.
\begin{enumerate}
\item
$g$ has at least one deepest commutator, and $[g,[g,k]]=1$ holds for
every deepest commutator $[g,k]$ of~$g$.
\item
Suppose that $k$~is any element of~$G$ such that $y = [g,k]$ satisfies
$[g,y]=1$.  Then the order of $y$~divides the order of~$g$.

In particular
if $y\neq 1$ and $g$ has order~$p$, then so does~$y$.
\end{enumerate}
\end{lemma}

\begin{proof}
Deepest commutators exist, and commute with~$g$ by maximality of~$r_0$.
If $1\neq y=[g,k]$ commutes with $g$,
then $g^k = gy$ and so
$(g^{p^n})^k = (gy)^{p^n} = g^{p^n}y^{p^n}$. So if $g^{p^n} = 1$ then
$y^{p^n} = 1$.
\end{proof}


\noindent
The concept ``deepest commutator of~$g$'' depends not only
on the element~$g$, but also on the ambient group~$G$.
When we need to stress the group~$G$ we shall
use the phrase ``deepest commutator in~$G$''\@.

\begin{notn}
Let $G$ be a $p$-group, $g$~an element of~$G$, and $N$ a normal subgroup
of~$G$. If $y,x \in L := C_G(g) N$ are such that $y = [g,x]$ is
a deepest commutator of~$g$ in~$L$, then we call $y$ a
\emph{$(G,N)$-locally deepest commutator} of~$g$.
A \emph{locally deepest} commutator is a $(G,N)$-locally deepest
commutator for some~$N$.
\end{notn}

\begin{rk}
Since $N$~is normal, $L$~is a group.
By the existence of deepest commutators, $g$~has a $(G,N)$-locally
deepest commutator if and only if $N \nleq C_G(g)$.
\end{rk}

\begin{lemma}
\label{lemma:Ndeepest}
If $y$ is an $(G,N)$-locally deepest commutator of $g \in G$, then $y =[g,x]$
for some $x \in N$.
\end{lemma}

\begin{proof}
We have that $y = [g,x']$ for some $x' \in C_G(g) N$, and so $x' = ax$ for some
$a \in C_G(g)$, $x \in N$. But then $g^{x'} = g^x$ and so $[g,x] = y$.
\end{proof}

\section{Late and last quadratics}
\label{section:late}

\noindent
Throughout this section, $G$ is a finite $p$-group and $V$ a faithful
$\f G$-module. Recall that an element $g \in G$ is called \emph{quadratic}
if its action on~$V$ has minimal polynomial $(X-1)^2$.

\begin{notn}
Suppose that $g \in G$ is quadratic.
We call $g$ \emph{late} quadratic if every locally deepest commutator of~$g$
is non-quadratic.
We call $g$ \emph{last} quadratic if every iterated commutator
$z = [g,h_1,h_2,\ldots,h_r]$ with $r \geq 1$ and $h_1,\ldots,h_r \in G$
is non-quadratic.
\end{notn}

\begin{rk}
An element $g \in G$ can only fail to have any locally deepest commutators if
$g \in Z(G)$, for it is only if $g \in Z(G)$ that
$N \leq C_G(g)$ holds for every $N \trianglelefteq G$.
In particular, every quadratic element in $Z(G)$ is late quadratic.

Note that every last quadratic element is late quadratic. We shall see in
Lemma~\ref{lemma:strange} that each late quadratic element has
elementary abelian normal closure in~$G$.  Hence all
the non-trivial iterated commutators of a last quadratic element have order~$p$.
\end{rk}

\begin{lemma}
\label{lemma:prestrange}
If $G$ has quadratic elements, then it has both
late quadratic and last quadratic elements.
\end{lemma}

\begin{proof}
As there are quadratic elements we may set
\[
t_0 := \max \{t \mid \text{$K_t(G)$ contains a quadratic element} \} \, .
\]
Pick a quadratic element $g \in K_{t_0}(G)$ and observe that every
commutator lies in $K_{t_0+1}(G)$. By the maximality of~$t_0$,
we see that $g$ is both last and late quadratic.
\end{proof}

\begin{lemma}
\label{lemma:strange}
Every late quadratic element $g$~of $G$
lies in $\Omega_1(Z(\yy(G)))$ and therefore satisfies
$\yy(G) \leq C_G(g)$.
\end{lemma}

\begin{proof}
We have already observed that all quadratic elements have order~$p$.
As $\yy(G)$ is a centric subgroup of~$G$ (Lemma~\ref{lemma:Yportmanteau}),
it suffices to show that each quadratic element $g$~of $G$ lies in
$C_G(\yy(G))$.
Let $1 = Y_0 \leq \cdots \leq Y_n = \yy(G)$ be a $Y$-series,
so
\begin{equation}
\mytag{eqn:Ystrange}
Y_r \trianglelefteq G \qquad \text{and} \qquad
[\Omega_1(C_G(Y_{r-1})),Y_r;2]=1 \, .
\end{equation}
We show that $g \in C_G(Y_r)$ by induction on~$r$. If $r=0$, this is clear.
Suppose $r \geq 1$ and
$g \in C_G(Y_{r-1})$. If $g \notin C_G(Y_r)$ then by
Lemma~\ref{lemma:Ndeepest}, $g$~has a $(G,Y_r)$-locally
deepest commutator $y = [g,x]$
with $x \in Y_r$. Then $[y,g]=1$ by Lemma~\ref{ol3:deepestCommutator},
and $[y,x] = [g,x,x]=1$ by
Eqn~\eqref{eqn:Ystrange}\@.
So, as $g$~is quadratic,
$y$ is too by Lemma~\ref{lemma:GHL}\@.
But $y$ cannot be quadratic,
since $y$~is a locally deepest commutator of the
late quadratic element $g \in G$\@.
\end{proof}

\begin{remark}
If we were attempting to prove Conjecture~\ref{conj:Y2} by induction
on~$\abs{G}$ then we could assume that it holds for $G/\yy(G)$.
By Theorem~\ref{thm:Y12} this would mean that $J(G) \leq \yy(G)$.
It would then follow from Lemma~\ref{lemma:strange} that every late quadratic
element lies in $\Omega_1(Z(J(G)))$,
the intersection of all the greatest rank elementary abelian subgroups of~$G$.
\end{remark}

\begin{lemma}
\label{lemma:lastQuadratic}
Let $H \leq G$ be the subgroup generated by all last quadratic
elements. Then every quadratic element of~$G$ lies in $C_G(H)$.
\end{lemma}

\begin{proof}
If not then $[g,h] \neq 1$ for some elements $g,h \in G$, with
$g$~last quadratic and $h$~quadratic.
By nilpotency $[g,h;r] \neq 1$ and
$[g,h;r+1]=1$ for some $r \geq 1$. Let $k = [g,h;r-1]$.
As the normal closure of $g$~in $G$ is abelian
(Lemma~\ref{lemma:strange}), all the iterated commutators of~$g$
commute, and so $[[k,h],k]=1$.
Also $[[k,h],h]=[g,h;r+1]=1$.
Lemma~\ref{lemma:GHL}
therefore says that $1 \neq [k,h]$ is quadratic, since $h$~is. But
$[k,h] = [g,h;r]$ cannot be quadratic, for $g$~is last quadratic.
Contradiction.
\end{proof}

\section{Powerful $p$-groups}
\label{section:powerful}
\noindent
We use L.~Wilson's paper~\cite{WilsonL:powerStructure} as a reference on
powerful $p$-groups. Recall from \cite[Definition 1.3]{WilsonL:powerStructure}
that for an odd prime~$p$, a finite $p$-group $G$ is called \emph{powerful} if
$G' \leq G^p := \langle g^p \mid g \in G \rangle$.
As in the proof of Proposition~\ref{prop:jammed}, recall that Timmesfeld's
Replacement Theorem tells us that if $V$~is a faithful $F$-module for~$G$,
then there is an elementary abelian subgroup $E$~of $G$ which is a quadratic
offender.
Our first result improves \cite[Theorem 6.2]{newol}\@.

\begin{proposition}
\label{prop:prePowerful}
Let $p$ be an odd prime and $G$ a finite $p$-group.
If $G' \cap \Omega_1(G)$ is abelian
then $G$ satisfies Conjecture~\ref{conj:Y2}\@.
\end{proposition}

\begin{proof}
Suppose that $V$ is a faithful $\f G$-module with no quadratic elements in
$\Omega_1(Z(G))$.
We must show that there are no quadratic offenders. So suppose that
the elementary abelian subgroup $E \leq G$ is a quadratic offender.
Then $[G',E]\neq 1$ by \cite[Theorem 1.5(2)]{newol}\@.
Moreover, $G'E\trianglelefteq G$. Hence $G'E\cap\Omega_1(G) \trianglelefteq G$.
By the Dedekind Lemma \cite[X.3]{Isaacs}, we have
$G'E\cap\Omega_1(G) = E (G'\cap\Omega_1(G))$, since $E\le\Omega_1(G)$.
Therefore $Z(E (G' \cap \Omega_1(G)))$ is an abelian normal subgroup of~$G$, and so
by \cite[Theorem 1.5(1)]{newol} it cannot contain~$E$. This means that
$G'\cap\Omega_1(G)$ is not centralized by $E$. So there must be an $a \in E$ such
that $[G'\cap\Omega_1(G),a] \neq 1$. But then $a \in \Omega_1(G)$, and so
$N := G' \langle a \rangle \cap\Omega_1(G)= \langle a \rangle (G'\cap\Omega_1(G))$
is a normal subgroup of $G$. Since $G' \cap \Omega_1(G)$ is abelian, a well-known
result (quoted as \cite[Lemma~6.1]{newol}) says that the commutator subgroup of
$N = \langle a \rangle (G'\cap\Omega_1(G))$ consists of commutators
in~$a$. But $N \trianglelefteq G$, and by choice of~$a$ we have $N' \neq 1$. Hence
$N'$~has nontrivial intersection with $\Omega_1(Z(G))$. So $1 \neq [a,x] \in \Omega_1(Z(G))$
for some $x \in G'\cap\Omega_1(G)$. Since $a$ is quadratic, Lemma~\ref{lemma:GHL}
says that $[a,x] \in \Omega_1(Z(G))$ is too.
This contradicts our assumption that there are no quadratic elements in $\Omega_1(Z(G))$.
So $G$ satisfies the conjecture.
\end{proof}

\begin{theorem}
\label{thm:powerful}
Let $p$ be an odd prime.
Every powerful $p$-group~$G$ satisfies Conjecture~\ref{conj:Y2}\@.
\end{theorem}

\begin{proof}
We refer to Wilson's paper~\cite{WilsonL:powerStructure}\@. His
Theorem 4.7 says that $G^p$ is powerful, and hence his Corollary 4.11
applied with $P = G$ shows that $\Omega_1(G^p)$ is powerful.
On the other hand, his Theorem 3.1 says that $\Omega_1(G)$~has exponent~$p$,
and so
$\Omega_1(G^p) = \Omega_1(G) \cap G^p$. So $\Omega_1(G) \cap G^p$ is powerful and of
exponent~$p$, which means it must be abelian. As $G' \leq G^p$, it follows that
$\Omega_1(G) \cap G'$ is abelian. So $G$ satisfies the conjecture, by
Proposition~\ref{prop:prePowerful} above.
\end{proof}

\subsection*{Acknowledgements}
We are grateful to Burkhard K\"ulshammer for suggesting that our methods
might also apply to $\GL_n(\f[q])$ in the coprime characteristic case.

We acknowledge travel funding from the Hungarian Scientific Research Fund
OTKA, grant 77-476\@.


\begin{thebibliography}{10}

\bibitem{BLO:survey}
C.~Broto, R.~Levi, and B.~Oliver.
\newblock The theory of {$p$}-local groups: a survey.
\newblock In P.~Goerss and S.~Priddy, editors, {\em Homotopy theory: relations
  with algebraic geometry, group cohomology, and algebraic $K$-theory}, volume
  346 of {\em Contemp. Math.}, pages 51--84. Amer. Math. Soc., Providence, RI,
  2004.

\bibitem{ChermakDelgado:Measuring}
A.~Chermak and A.~Delgado.
\newblock A measuring argument for finite groups.
\newblock {\em Proc. Amer. Math. Soc.}, 107(4):907--914, 1989.

\bibitem{GAP4}
The GAP~Group.
\newblock {\em {GAP -- Groups, Algorithms, and Programming, Version 4.4.12}},
  2008.
\newblock \verb+(http://www.gap-system.org)+.

\bibitem{oliver}
D.~J. Green, L.~H\'ethelyi, and M.~Lilienthal.
\newblock On {O}liver's $p$-group conjecture.
\newblock {\em Algebra Number Theory}, 2(8):969--977, 2008.
\newblock arXiv:0804.2763v2 [math.GR].

\bibitem{newol}
D.~J. Green, L.~H\'ethelyi, and N.~Mazza.
\newblock On {O}liver's $p$-group conjecture: {II}.
\newblock {\em Math. Ann.}, 347(1):111--122, May 2010.
\newblock DOI: 10.1007/s00208-009-0435-4 \quad arXiv:0901.3833v1 [math.GR].

\bibitem{Huppert:I}
B.~Huppert.
\newblock {\em Endliche {G}ruppen. {I}}.
\newblock Die Grundlehren der Mathematischen Wissenschaften, Band 134.
  Springer-Verlag, Berlin, 1967.

\bibitem{Isaacs}
I.~M. Isaacs.
\newblock {\em Finite group theory}, volume~92 of {\em Graduate Studies in
  Mathematics}.
\newblock American Mathematical Society, Providence, RI, 2008.

\bibitem{Lynd:2subnormal}
J.~Lynd.
\newblock 2-subnormal quadratic offenders and {O}liver's $p$-group conjecture.
\newblock Submitted, Feb. 2010.
\newblock Available at: \\
  \verb+http://www.math.ohio-state.edu/~jlynd/research/oliver.class.pdf+.

\bibitem{MeierfStellm:OtherPGV}
U.~Meierfrankenfeld and B.~Stellmacher.
\newblock The other {${\scr P}(G,V)$}-theorem.
\newblock {\em Rend. Sem. Mat. Univ. Padova}, 115:41--50, 2006.

\bibitem{Oliver:MartinoPriddyOdd}
B.~Oliver.
\newblock Equivalences of classifying spaces completed at odd primes.
\newblock {\em Math. Proc. Cambridge Philos. Soc.}, 137(2):321--347, 2004.

\bibitem{Oliver:MartinoPriddyEven}
B.~Oliver.
\newblock Equivalences of classifying spaces completed at the prime two.
\newblock {\em Mem. Amer. Math. Soc.}, 180(848):vi+102, 2006.

\bibitem{Weir:Classical}
A.~J. Weir.
\newblock Sylow {$p$}-subgroups of the classical groups over finite fields with
  characteristic prime to {$p$}.
\newblock {\em Proc. Amer. Math. Soc.}, 6:529--533, 1955.

\bibitem{Weir:GeneralLinear}
A.~J. Weir.
\newblock Sylow {$p$}-subgroups of the general linear group over finite fields
  of characteristic {$p$}.
\newblock {\em Proc. Amer. Math. Soc.}, 6:454--464, 1955.

\bibitem{WilsonL:powerStructure}
L.~Wilson.
\newblock On the power structure of powerful {$p$}-groups.
\newblock {\em J. Group Theory}, 5(2):129--144, 2002.

\end{thebibliography}

\end{document}